\documentclass[english,11pt, a4paper]{article}
 
\title{Cauchy Integral, Fractional Sobolev Spaces\\ and Chord-Arc Curves}

\usepackage[T1]{fontenc}

\usepackage[hypertexnames=false]{hyperref}
\usepackage[utf8]{inputenc}

\usepackage{mathtools}
\usepackage{lmodern}
\usepackage{amsthm,amsmath,amsfonts,amssymb}
\usepackage{graphicx}
\usepackage{enumerate,enumitem}
\usepackage{authblk}
\usepackage{cite}

\let\OLDthebibliography\thebibliography
\renewcommand\thebibliography[1]{
  \OLDthebibliography{#1}
  \setlength{\parskip}{0pt}
  \setlength{\itemsep}{0pt plus 0.3ex}
}

\usepackage[activate={true,nocompatibility},final,tracking=true,kerning=true,spacing=true,factor=1100,stretch=15,shrink=15]{microtype}
\microtypecontext{spacing=nonfrench}

\allowdisplaybreaks

\makeatletter
\numberwithin{equation}{section}
\theoremstyle{plain}
\newtheorem{theo}{\protect\theoremname}
  \theoremstyle{plain}
  
    \newtheorem{prop}[theo]{\protect\propname}

\newtheorem{df}[theo]{Definition}

\numberwithin{theo}{section}

\newtheorem{cor}[theo]{Corollary}

\newtheorem{rem}[theo]{Remark}

\makeatother

\usepackage{babel}
\providecommand{\propname}{Proposition}

\providecommand{\lemmaname}{Lemma}
\providecommand{\theoremname}{Theorem}

\usepackage{hyperref}
\hypersetup{
     colorlinks=true, % make the links colored
   linkcolor=black, % color TOC links in blue
    urlcolor=black, % color URLs in black
    linktoc=all % 'all' will create links for everything in the TOC
}

 \author{
 Huaying Wei\thanks{Center for applied Mathematics, Tianjin University, Weijin Road 92, Tianjin, 300072, P.R. China, email: \url{hywei@tju.edu.cn} } \, and Michel Zinsmeister\thanks{Institut Denis Poisson,  Orl\'eans, 45067,  France, email: \url{zins@univ-orleans.fr}}}
 \date{\today}

\begin{document}

\maketitle

\begin{abstract}
     Let $\Gamma$ be a bounded Jordan curve and  $\Omega_i,\Omega_e$ its two complementary components. For $s\in(0,1)$ we define $\mathcal{H}^s(\Gamma)$ as the set of functions $f:\Gamma\to \mathbb C$ having harmonic extension $u$ in $\Omega_i\cup \Omega_e$ such that
$$ \iint_{\Omega_i\cup \Omega_e} |\nabla u(z)|^2 d(z,\Gamma)^{1-2s}  dxdy<+\infty.$$
If $\Gamma$ is further assumed to be rectifiable we define $H^s(\Gamma)$ as the space of measurable functions $f:\Gamma\to \mathbb C$ such that 
$$\iint_{\Gamma\times \Gamma}\frac{|f(z)-f(\zeta)|^2}{|z-\zeta|^{1+2s}} d\sigma(z)d\sigma(\zeta)<+\infty.$$
When $\Gamma$ is the unit circle these two spaces coincide with the homogeneous fractional Sobolev space defined via Fourier series. For a general rectifiable curve these two spaces need not coincide and our first goal is to investigate the cases of equality: while the chord-arc property is the necessary and sufficient condition for equality in the classical case of $s=1/2$  (see \cite{WZ}), this is no longer the case for general $s\in (0,1)$. We show however that equality holds for Lipschitz curves. 

The second goal involves the  Plemelj-Calder\'on  problem. If $\Gamma$ is rectifiable, following Calder\'on,  it reduces to the boundedness of the Cauchy singular integral operator on the space $H^s(\Gamma)$: we prove it for chord-arc curves using interpolation.
In general, the problem may be stated as follows: given $f\in \mathcal{H}^s(\Gamma)$, is it possible to write
$$ f=F_i|_{\Gamma}+F_e|_{\Gamma},$$
with $F_{i,e}$ holomorphic in $\Omega_{i,e}$ and
$$\iint_{\Omega_{i,e}} | F_{i,e}'(z)|^2 d(z,\Gamma)^{1-2s}  dxdy\le C\|f\|_{\mathcal{H}^s(\Gamma)}^2?$$
Depending on the the fractal nature of $\Gamma$ we show the existence of an interval of values of $s$ for which the property holds. For chord-arc curves this interval is $(0,1)$ and the results for $H^s(\Gamma)$ and $\mathcal{H}^s(\Gamma)$ coincide for Lipschitz curves.

\bigskip 

   \noindent \textbf{Keywords:} Cauchy integral, fractional Sobolev spaces, chord-arc curves, Dirichlet spaces, interpolation
   
   \noindent \textbf{Mathematics Subject Classification 2020:} 42B20, 46E35, 31A05, 30H35

\end{abstract}

%\newpage
\tableofcontents
%\newpage

\section{Introduction}

It is an old problem , given a bounded Jordan curve $\Gamma$ in the plane and a (complex valued) function $f$ defined on $\Gamma$, to find two holomorphic functions $G_i,G_e$ defined on the interior-connected and the exterior-connected components of $\Gamma$ such that 
$$ f=G_i|_\Gamma+G_e|_\Gamma,$$
the boundary traces being defined in some sense.

This problem was actually raised and solved by Sokhotski in 1873  before being rediscovered by Plemelj as a main ingredient of his attempt to solve Hilbert's 22th problem in 1908 (\cite{Ple}). Assume first that $\Gamma$ is  smooth and $f$ is analytic on $\Gamma$. 
The idea is to use the Cauchy integral and define
$$ F(z)=\frac{1}{2\pi i}\int_\Gamma\frac{f(\zeta)}{\zeta-z}d\zeta, \qquad z \in \mathbb C\setminus\Gamma.$$
This defines a function which is holomorphic outside the curve $\Gamma$. Moreover,  this function has boundary values on $\Gamma$ from inside and outside, and one has  Plemelj's formula on $\Gamma$: 
\begin{align*}
    F_i|_{\Gamma} = Tf + 1/2 f;\\
    F_e|_{\Gamma} = Tf-1/2 f,
   \end{align*}
where the Cauchy integral 
$$ Tf(z)= \frac{1}{2\pi i}\text{P.V.}\int_\Gamma \frac{f(\zeta)}{\zeta-z}d\zeta = \frac{1}{2\pi i} \lim\limits_{\varepsilon\to 0}\int_{|\zeta-z|>\varepsilon} \frac{f(\zeta)}{\zeta-z}d\zeta.$$ 
The problem is then solved with $G_i=F_i$ and $G_e=-F_e$.

Of course the result remains true for much more general curves and functions. Subsequent generalizations relax the smoothness requirements on the curve $\Gamma$ and the function $f$: In order for  Plemelj's formula to hold true one indeed only needs $\Gamma$ to be rectifiable and $f\in L^1(\Gamma, d\sigma)$. The modern approach of this problem is then to solve the problem in a given space: Calder\'on (\cite{Cal})
asked for instance for which rectifiable curves it is true that for any $f\in L^2(\Gamma,d\sigma)$ the Plemelj functions $F_i|_\Gamma$ and $F_e|_\Gamma$ are also in  $ L^2(\Gamma,d\sigma)$? Equivalently, when is the operator $T$ bounded on  $f\in L^2(\Gamma,d\sigma)$? This problem has been solved by Coifman-McIntosh-Meyer (\cite{CMM}) for Lipschitz curves (Calder\'on  had previously solved it for curves with small Lipschitz bound (\cite {Cal})) and then the last word was given by David (\cite{Dav}) who proved that $T$ is bounded on  $f\in L^2(\Gamma,d\sigma)$ if and only if $\Gamma$ is Ahlfors-regular, meaning the length of the part of $\Gamma$ lying in a disk of radius $r$ is no bigger than $Cr$ for some  $C$ independent of the disk.

Now the Plemelj-Calder\'on problem may be addressed for curves that are not necessarily  rectifiable. An example has been given in \cite{Zi1} where the Plemelj-Calder\'on problem for H\"older classes has been raised and solved in some particular cases. The idea is to replace the Cauchy integral by the identity
$$ f(z)=-\frac 1\pi \iint_\mathbb{C}\frac{\bar{\partial}f(\zeta)}{\zeta-z}d\xi d\eta$$ 
which is true for any test function $f\in C_c^\infty(\mathbb C)$ (infinitely differentiable functions in $\mathbb C$ with compact support), and from which it follows that, in the case of rectifiable boundary,
\begin{align*}\label{ple}
    F_i(z)&=-\frac 1\pi \iint_{\Omega_e}\frac{\bar{\partial}f(\zeta)}{\zeta-z}d\xi d\eta,\qquad z \in \Omega_i;\\
    F_e(z)&=\frac 1\pi \iint_{\Omega_i}\frac{\bar{\partial}f(\zeta)}{\zeta-z}d\xi d\eta,\qquad z \in \Omega_e
\end{align*}
where $\Omega_{i,e}$ stand for the interior and exterior components of $\Gamma$. Later Astala (\cite{Ast}) obtained more precise results for the H\"older classes: these will be discussed in Section 3 where we will draw the parallel with fractional Sobolev spaces.

The principal aim of this paper is to investigate the Plemelj-Calder\'on problem for the fractional Sobolev spaces. There will be two angles of attack, depending on whether the curve is rectifiable or not, following the preceding lines. In order to understand these two strategies, let us start with the simplest case of $\Gamma$ being the unit circle $\mathbb T$. Within this framework we adopt the powerful tool of Fourier analysis: for $0 \leq s<1$ we define the  fractional Sobolev space $H^s(\mathbb T)$ as the set of $L^2(\mathbb T)$-functions $f$ such that the series
$$ \sum\limits_{n\in \mathbb Z}|n|^{2s}|\hat{f}(n)|^2$$
converges. 
 An application of Parseval's formula shows that this series is equal, up to a multiplication constant, to $\|f\|_{L^2(\mathbb T)}^2$ for $s = 0$, and to 
$$ \|f\|_{H^s(\mathbb T)}^2=\iint_{\mathbb T\times\mathbb T}\frac{|f(z)-f(\zeta)|^2}{|z - \zeta|^{1+2s}} |dz||d\zeta|$$
for $0 < s < 1$. 
We will call this last quantity the (square of the) Douglas norm.   
Another point of view is what we will call the Littlewood-Paley one: it can be computed directly using Parseval's formula (see \cite{AS}, \cite{Ste} and also \cite{Tay}) that for $0 \leq s < 1$ the series is also equivalent to
$$ \|u_i\|_{\mathcal{H}^s(\mathbb{D}_i)}^2 = \iint_{\mathbb {D}_i}|\nabla u_i(z)|^2(1-|z|^2)^{1-2s} dxdy.$$
Notice that when $s = 0$ if $1-|z|^2$ in the integrand is replaced equivalently by $1-|z|$ then this integral is just the mean value of the Littlewood-Paley $\mathsf g$-function of $u_i$ on the unit circle $\mathbb T$ multiply by $2\pi$. 
Using Schwarz reflection, we can see that is equivalent to
$$
   \|u_e\|_{\mathcal{H}^s(\mathbb D_e)}^2 = \iint_{\mathbb D_e} |\nabla u_e(z)|^2(|z|^2-1)^{1-2s}dxdy
    $$
where $u_i$ and $u_e$ stand for the harmonic extensions of $f$ to the interior domain $\mathbb D_i$ and the exterior domain $\mathbb D_e$ of the unit circle $\mathbb T$, respectively. Notice that for $s = 1/2$ these equivalences reduce to equalities (up to a multiplication constant).

For a general Jordan curve $\Gamma$, one may define similarly $\mathcal{H}^s(\Omega_{i,e})$ as being the set of harmonic functions $u$ defined on  $\Omega_{i,e}$ by its (square) norm
$$ \|u\|_{\mathcal{H}^s(\Omega_{i,e})}^2=\iint_{\Omega_{i,e}}|\nabla u(z)|^2 d(z,\Gamma)^{1-2s}dxdy,$$
where $d(z,\Gamma)$ stands for the distance from $z$ to $\Gamma$. We would like to identify the elements $u\in \mathcal{H}^s(\Omega_{i,e})$ with their ``boundary values''. This can be done if $\Gamma$ is a quasicircle and  $s \in (1/2, 1)$. 
\begin{prop} 
Let $\Omega$ be a domain bounded by a bounded Jordan curve $\Gamma$. For any $u\in \mathcal{H}^s(\Omega)$ with $s \in (1/2, 1)$ 
there exists a constant $C>0$ such that 
$$|\nabla u(z)|\le Cd(z,\Gamma)^{s-3/2}, \qquad z \in \Omega.$$
\end{prop}
\begin{proof}
    Let $z\in \Omega$ and $D$ is the disk $D(z,d(z,\Gamma)/2)$. By the mean value property of harmonic functions one can write
$$ \nabla u(z)=\frac{1}{|D|}\iint_D\nabla u(\zeta)d\xi d\eta$$
(Here and in the sequel we write $|D|$ for the Lebesgue measure of $D$.) 
so that, by Cauchy-Schwarz inequality,
\begin{align*}
|\nabla u(z)|^2&\le \frac{4}{\pi d(z, \Gamma)^2}\iint_D|\nabla u(\zeta)|^2d\xi d\eta\\
&\le  Cd(z, \Gamma)^{2s-3}\iint_D|\nabla u(\zeta)|^2 d(\zeta,\Gamma)^{1-2s} d\xi d\eta\\
& \leq C\|u\|_{\mathcal{H}^s(\Omega)}^2d(z,\Gamma)^{2s-3}.
\end{align*}
\end{proof}
When $\Gamma$ is a quasicircle we may apply the following result due to Gehring-Martio (\cite{GM}). 
\begin{prop} Let $\Omega$ be a  domain bounded by a quasicircle $\Gamma$ and $\alpha \in (0,1]$. The following are equivalent: 
\begin{enumerate}
\item[\rm(1)] $\forall z\in \Omega,\; |\nabla u(z)|\le Cd(z,\Gamma)^{\alpha - 1}$ for some constant $C > 0$;
\item[\rm(2)] $u\in \Lambda^\alpha(\bar{\Omega})$ .
\end{enumerate}
\end{prop}
\noindent In this statement $\Lambda^\alpha(\bar{\Omega})$ stands for the space of H\"older functions of order $\alpha$. This settles the case $s \in (1/2, 1)$: boundary values of functions in $\mathcal{H}^s(\Omega_{i,e})$ are well-defined and they characterize $u$, as being the unique harmonic extension. We may now define the space $\mathcal{H}^s(\Gamma)$ for this range $s\in (1/2,1)$ as being the space of 
$(s-1/2)$-H\"older functions $f$ on $\Gamma$ whose harmonic extension $u_{i,e}$ to $\Omega_{i,e}$ belongs to $\mathcal{H}^s(\Omega_{i,e})$, and the space  $\mathcal{H}^s(\Gamma)$ is assigned the natural Hilbert norm $\|\cdot\|_{\mathcal{H}^s(\Gamma)}$ so that $\|f\|^2_{\mathcal{H}^s(\Gamma)} = \Vert u_i\Vert^2_{\mathcal{H}^s(\Omega_i)} + \Vert u_e\Vert^2_{\mathcal{H}^s(\Omega_e)}$.

The case $s=1/2$ has been treated by Schippers-Staubach  when $\Gamma$ is a quasicircle (\cite{Sc1}):  the boundary values of functions in  $\mathcal{H}^{1/2}(\Omega_{i,e})$ is well-defined by taking radial limits (see Section 2 for a more explicit explanation). We may also define the space  $\mathcal{H}^{1/2}(\Gamma)$ as being the space of functions $f$ on $\Gamma$ whose harmonic extension $u_{i,e}$ to $\Omega_{i,e}$ belongs to $\mathcal{H}^{1/2}(\Omega_{i,e})$. It is a quite large class of functions.  In particular, by Dirichlet's principle, it is the closure of the space $D(\Gamma)$ of restrictions to $\Gamma$ of $C_c^\infty(\mathbb C)$ functions under the Hilbert norm in $\mathcal{H}^{1/2}(\Gamma)$, defined as the case of $s\in (1/2,1)$ (see \cite{WZ}).

For $0 \leq s<1/2$ the situation is less clear if $\Gamma$ is not rectifiable. 
In order to overcome this difficulty,  we consider the space $D(\Gamma)$ and define $D(\Omega_{i,e})$ as the set of harmonic extensions of functions in $D(\Gamma)$ to $\Omega_{i,e}$.
\begin{prop} Let $\Omega$ be a domain bounded by a bounded Jordan curve $\Gamma$. 
If $0 \leq s<1/2$ then 
$$D(\Omega)\subset\mathcal{H}^s(\Omega).$$
\end{prop}
\begin{proof}
    Let $F$ be a $C_c^\infty(\mathbb C)$ function such that $f=F|_\Gamma\in D(\Gamma)$. Then obviously 
$$\iint_{\Omega}|\nabla F(z)|^2dxdy<\infty.$$
By the Dirichlet principle, we then have that $u\in \mathcal{H}^{1/2}(\Omega)\subset \mathcal{H}^s(\Omega)$ where $u$ is the harmonic extension of $f$ to $\Omega$.
\end{proof}
Inspired by the case $s = 1/2$ we finally define $\mathcal{H}^s(\Gamma)$ for the range $s \in [0, 1/2)$ as the completion of $D(\Gamma)$ for the naturally associated Hilbert norm as above. It is easy to see the inclusion relations $\mathcal H^{s_1}(\Gamma)\subset \mathcal H^{1/2}(\Gamma) \subset \mathcal H^{s_2}(\Gamma)$ for any $0 \leq  s_2 < 1/2 < s_1 < 1$.

When the curve $\Gamma$ is furthermore assumed to be rectifiable we may also define the space $H^s(\Gamma)$, $0 < s < 1$, through (the square of) its Douglas-type norm
$$ \|f\|_{H^s(\Gamma)}^2=\iint_{\Gamma\times\Gamma}\frac{|   f(z)-f(\zeta)|^2}{|z-\zeta|^{1+2s}}d\sigma(z)d\sigma(\zeta)$$
where $d\sigma$ denotes the arc-length measure. 

Recall that when $\Gamma = \mathbb T$ the norms $\|f\|_{H^s(\mathbb T)}$, $\|u_i\|_{\mathcal{H}^s(\mathbb D_i)}$ and $\|u_e\|_{\mathcal{H}^s(\mathbb D_e)}$ are comparable for $0<s<1$, but this comparability has no reason to hold for a general rectifiable Jordan curve $\Gamma$. One of the purposes of this paper is to investigate the comparable case, in connection with  the Plemelj-Calder\'on problem. Let us be more precise about what we mean by the Plemelj-Calder\'on problem in the two settings introduced above:
\begin{enumerate}
\item[(a)] If $\Gamma $ is rectifiable, the problem reduces to proving that the operator $T$ from  Plemelj's formula is bounded on $H^s(\Gamma), 0<s<1$.
\item[(b)] In the general case we ask whether there exists $C>0$ such that if $f\in\mathcal{H}^s(\Gamma)$, $0<s<1$, 
$$ \|F_{i,e}\|_{\mathcal{H}^s(\Omega_{i,e})}\le C\|f\|_{\mathcal{H}^s(\Gamma)}.$$
\end{enumerate}
In Section $2$ we will consider the special case $s=1/2$ of setting (b) with the use of the Beurling operator. In Section $3$ we will consider the general case of (b) using Astala's generalization of Ahlfors-regularity.  More precisely,  we will define a kind of Minkowski content $h_\delta(\Gamma)$ for dimension $1\le \delta\le 2$, call $\Gamma$ to be $\delta$-regular if there exists $C>0$ such that $h_\delta(\Gamma\cap D(z,r))\le Cr^\delta$ for all $z\in \Gamma$ and $r>0$ and define $h(\Gamma)$ as being the infinimum of the $\delta's$ such that $\Gamma $ is $\delta$-regular. We prove then
\begin{theo} 
Let $\Gamma$ be a quasicircle. Then the Calder\'on-Plemelj property is true for $\mathcal{H}^s(\Gamma)$ if
$$ \frac{h(\Gamma)-1}{2}<s<\frac{3-h(\Gamma)}{2}.$$
\end{theo}
\noindent In Section $4$ we come to setting (a) and prove
\begin{theo} The Calder\'on-Plemelj property holds for $H^s(\Gamma),\,0 < s < 1$, when $\Gamma$ is a chord-arc curve.
\end{theo}
\noindent In Section $5$ we study the intersection between the two results. In other words,  we discuss for which chord-arc curves the two theorems have the same conclusion. We have the following result. See Theorem \ref{interpol} for a more explicit and precise statement. 

\begin{theo}
Let $\Omega$ be a  domain bounded by a chord-arc curve $\Gamma$ and $\varphi$ its Riemann mapping. If $\Gamma$ is such that the Muckenhoupt weight $|\varphi'|$ has $A_2$ on $\mathbb T$ then, for $0 < s < 1$, $\mathcal{H}^s(\Gamma) = H^s(\Gamma)$.
\end{theo}

\noindent \textbf{Conventions:} Throughout the paper, we  deal with the bounded Jordan curve $\Gamma$ in the complex plane, and  denote the interior and exterior domains of $\Gamma$ by $\Omega_i$ and $\Omega_e$ (and that of the unit circle $\mathbb T$ by $\mathbb D_i$ and $\mathbb D_e$), respectively. $\Omega$  stands for $\Omega_i$ or $\Omega_e$ (and $\mathbb D$  stands for $\mathbb D_i$ or $\mathbb D_e$).

\section{The Dirichlet space on quasidisks}
Before we consider the general fractional Sobolev spaces let us focus on the special case $s=1/2$ which is by many means remarkable. Recall that if $\Omega$ is a domain bounded by a Jordan curve $\Gamma$ the Dirichlet space $\mathcal{H}^{1/2}(\Omega)$ is the set of harmonic functions $u:\Omega\to \mathbb C$ such that the $L^2(\Omega)$-norm of the gradient vector $\nabla u(x, y) = (u_x, u_y)$ finite: 
$$ \iint_\Omega |\nabla u(x,y)|^2dxdy<\infty.$$
The complex notation is more convenient for our purpose. Noting that $u_{\bar z} = (u_x + iu_y)/2$ and $u_z = (u_x - i u_y)/2$ where $z = x+iy$ we have that $|\nabla u(x, y)|^2 = 2 (|u_z|^2 + |u_{\bar z}|^2)$. 
The Dirichlet space is a Hilbert space of functions modulo constants which is conformally invariant: if $\varphi$ is a conformal map from  $\mathbb D_i$ onto $\Omega_i$
(or $\psi$ from $\mathbb D_e$ onto $\Omega_e$)   then the mapping $u\mapsto u\circ \varphi$ (or $u\mapsto u\circ \psi$) is an isometry between $\mathcal{H}^{1/2}(\Omega_{i})$ and $\mathcal{H}^{1/2}(\mathbb D_{i})$ (or between $\mathcal{H}^{1/2}(\Omega_{e})$ and $\mathcal{H}^{1/2}(\mathbb D_{e})$), which is isomorphic  to $H^{1/2}(\mathbb T)$.

Recall that a quasidisk is the image of the unit disk $\mathbb D_i$ by a quasiconformal homeomorphism of the plane, and its boundary curve is called a quasicircle. An homeomorphism $\Phi$ of the plane is called quasiconformal if its gradient in the sense of distribution is in $L^2_{loc}(\mathbb C)$ and if in addition there exists a constant $k<1$ such that
$$ \frac{\partial\Phi }{\partial \bar z} = \mu(z) \frac{\partial\Phi }{\partial z}$$
for an essentially uniformly bounded function $\mu \in L^{\infty}(\mathbb C)$ bounded by $k$. 
If $\varphi$ and $\psi$ are the conformal isomorphisms from $\mathbb D_i$ onto  $\Omega_i$ and from $\mathbb D_e$ onto $\Omega_e$, respectively,  then $h=\psi^{-1}\circ\varphi$ is an homeomorphism of $\mathbb T$ which is quasisymmetric in the sense that there exists a constant $C > 0$ such that 
$$
C^{-1} \leq \frac{\vert h(e^{i(t+\alpha)})-h(e^{it})\vert}{\vert h(e^{it})-h(e^{i(t-\alpha)})\vert}\leq C
$$
for every $t\in \mathbb R$ and $-\pi/2<\alpha\leq\pi/2$.  Moreover, every quasisymmetry arises in this way from some quasidisk. For these results on quasiconformal theory, we refer to \cite{Ahl} for details. The quasisymmetries of $\mathbb T$ form a group, and this group coincides with the group of homeomorphisms $h$ of $\mathbb T$ such that $f\mapsto f\circ h$ is an isomorphism of $H^{1/2}(\mathbb T)$ (\cite{NaS}). 

Noting that the quasisymmetry $h$ of $\mathbb T$ may be singular, we see  harmonic measures on $\Omega_i$ and $\Omega_e$ might be incomparable, that possibly causes  the distinction between boundary values obtained from $\mathcal H^{1/2}(\Omega_i)$ and $\mathcal{H}^{1/2}(\Omega_e)$. In \cite{Sc1}, Schippers-Staubach showed that is not the case. 
We outline their solution for the sake of completeness. On the one hand, any function $u \in \mathcal H^{1/2}(\mathbb D_{i})$ has radial limits for all  $e^{i\theta} \in \mathbb T$ except on a Borel set of logarithmic capacity zero. It follows from the conformal invariance of Dirichlet spaces that any $u \in \mathcal H^{1/2}(\Omega_{i})$ has radial limits on $\Gamma$ except on a Borel set, which is the image of a subset of $\mathbb T$ of logarithmic capacity zero under the conformal map $\varphi$, and the same assertion holds for the function in  $\mathcal H^{1/2}(\Omega_e)$. On the other hand, the quasisymmetry takes Borel sets of logarithmic capacity zero to Borel sets of logarithmic capacity zero. Notice that a set of logarithmic capacity zero on $\mathbb T$ has harmonic measure zero with respect to both $\mathbb D_i$ and $\mathbb D_e$. From this discussion it follows that if $\Gamma$ is a quasicircle then for any $u_i \in \mathcal H^{1/2}(\Omega_i)$, the unique harmonic extension $u_e$ in $\mathcal H^{1/2}(\Omega_e)$ of boundary function of $u_i$ has the same boundary values as $u_i$, except on a set whose harmonic measure is zero with respect to both 
 $\Omega_i$ and $\Omega_e$. Moreover, using the isomorphism of the operator $f\mapsto f\circ h$ of $H^{1/2}(\mathbb T)$, it can be shown that the transmission $u_i\mapsto u_e$ from $\mathcal H^{1/2}(\Omega_i)$ onto  $\mathcal H^{1/2}(\Omega_e)$  is a bounded isomorphism.

We can now address the Plemelj-Calder\'on problem for $\mathcal{H}^{1/2}(\Gamma)$. The argument is broken into two steps. In the first we consider the function in the space $D(\Gamma)$, the restriction space to $\Gamma$ of $C_c^{\infty}(\mathbb C)$. In the second the general case is done using an approximate process.

For any test function $f$, it holds that 
\begin{equation}\label{convolution}
   f(z) = -\frac{1}{\pi}\iint_{\mathbb C}\frac{\bar\partial f(\zeta)}{\zeta - z}d\xi d\eta, \qquad z \in \mathbb C.
\end{equation}
Indeed, 
\begin{equation*}\label{Dirac}
    \frac{1}{\pi}\iint_{\mathbb C}\frac{\bar\partial f(\zeta)}{z - \zeta}d\xi d\eta = \bar\partial f(z) \ast \left(\frac{1}{\pi z}\right) = \bar\partial \left(\frac{1}{\pi z}\right)\ast f(z) = \delta_0\ast f(z) = f(z)
\end{equation*}
where $\ast$ stands for convolution, and $\delta_0$ denotes Dirac function.  
Define
\begin{equation}\label{Fi1}
    F_i(z) = -\frac{1}{\pi}\iint_{\Omega_e}\frac{\bar\partial f(\zeta)}{\zeta - z}d\xi d\eta, \qquad z \in \Omega_i,
\end{equation}
\begin{equation}\label{Fe1}
    F_e(z) = \frac{1}{\pi}\iint_{\Omega_i}\frac{\bar\partial f(\zeta)}{\zeta - z}d\xi d\eta, \qquad z \in \Omega_e.
\end{equation}
It is easily seen that $F_i$ and $F_e$ are holomorphic functions in $\Omega_i$ and $\Omega_e$, respectively with $F_e(z) = O(\frac{1}{|z|})$ at $\infty$. Furthermore, it can be shown that 
both $F_i$ and $F_e$ define continuous functions at each $z \in \mathbb C$. 
It follows from \eqref{convolution} that on $\Gamma$,  
\begin{equation}\label{split}
    f  = F_i|_{\Gamma} - F_e|_{\Gamma}. 
\end{equation}
Here, we have used that the quasicircle $\Gamma$ has zero area with respect to the two-dimensional Lebesgue measure. 

Let $u_i$ as before be the harmonic extension of $f$ to $\Omega_i$. Noting that in $\Omega_i$ the infinitely differentiable  function $f - u_i$ vanishes on the boundary $\Gamma$, by the formula of integration by parts we have  
\begin{equation}\label{W11}
    \iint_{\Omega_i}\frac{\bar\partial(f - u_i)(\zeta)}{\zeta - z} d\xi d\eta = (-1)\iint_{\Omega_i}(f - u_i)(\zeta)\bar\partial \Big(\frac{1}{\zeta - z}\Big)d\xi d\eta = 0,
\end{equation}
and then 
\begin{equation}\label{Fe2}
   F_e(z) = 
     \frac{1}{\pi}\iint_{\Omega_i}\frac{\bar\partial u_i(\zeta)}{\zeta - z}d\xi d\eta
\end{equation}
for $z \in \Omega_e$, 
which implies that $F_e$  does indeed depend only on $f|_{\Gamma}$ and not on the specific extension. Similarly, we have
\begin{equation}\label{Fi2}
     F_i(z) = -\frac{1}{\pi}\iint_{\Omega_e}\frac{\bar\partial u_e(\zeta)}{\zeta - z}d\xi d\eta
\end{equation}
for $z \in \Omega_i$, where $u_e$ is as before the harmonic extension of $f$ to $\Omega_e$. From this expression, it follows that 
$$
F_i'(z) = -\frac{1}{\pi}\iint_{\Omega_e}\frac{\bar\partial u_e(\zeta)}{(\zeta - z)^2}d\xi d\eta 
= B(\bar\partial u_e \chi_{\Omega_e})(z).
$$
Here, $\chi_{\Omega_e}$ is the characteristic function of the domain $\Omega_e$ and $B$ denotes the Beurling operator, i.e., the convolution with $-\frac 1\pi \text{P.V.}\frac{1}{z^2}$ which is an isometry of $L^2(\mathbb C)$. Then, 
\begin{align*}
   \Vert F_i \Vert_{\mathcal{H}^{1/2}(\Omega_i)} = \Vert F_i' \Vert_{L^2(\Omega_i)} &= \Vert B(\bar\partial u_e \chi_{\Omega_e}) \Vert_{L^2(\Omega_i)}\\
    &\leq  \Vert B(\bar\partial u_e \chi_{\Omega_e}) \Vert_{L^2(\mathbb C)} = \Vert \bar\partial u_e \chi_{\Omega_e}  \Vert_{L^2(\mathbb C)}\\
    & = \Vert \bar\partial u_e \Vert_{L^2(\Omega_e)} \leq \Vert u_e \Vert_{\mathcal{H}^{1/2}(\Omega_e)} \approx \|f\|_{\mathcal H^{1/2}(\Gamma)}. 
\end{align*}
Here and below, the notation $A\approx B$ denotes that $A$ and $B$ are comparable. On the last step we have used that the operator $\mathcal H^{1/2}(\Omega_i)\to \mathcal H^{1/2}(\Omega_e)$, sending $u_i$ to $u_e$, is a bounded isomorphism with respect to $\|\cdot\|_{\mathcal H^{1/2}(\Omega_i)}$ and $\|\cdot\|_{\mathcal H^{1/2}(\Omega_e)}$ if and only if the curve $\Gamma$ is a quasicircle (see \cite{WZ}). 
Similarly, it follows from \eqref{Fe2} that $\Vert F_e \Vert_{\mathcal H^{1/2}(\Omega_e)} \leq \Vert u_i \Vert_{\mathcal H^{1/2}(\Omega_i)} \approx \|f\|_{\mathcal H^{1/2}(\Gamma)}$.

Let now $f$ be any function of $\mathcal H^{1/2}(\Gamma)$. Define the functions $F_i$ and $F_e$ as in \eqref{Fi2} and \eqref{Fe2}. 
It can be shown that these two functions are still  well-defined in this general case 
and holomorphic in the domains of definition. By the same reasoning as above, we  have
 $\Vert F_i \Vert_{\mathcal H^{1/2}(\Omega_i)} \leq \Vert u_e \Vert_{\mathcal H^{1/2}(\Omega_e)} \approx \|f\|_{\mathcal H^{1/2}(\Gamma)}$ and 
 $\Vert F_e \Vert_{\mathcal H^{1/2}(\Omega_e)} \leq \Vert u_i \Vert_{\mathcal H^{1/2}(\Omega_i)} \approx \|f\|_{\mathcal H^{1/2}(\Gamma)}$. Recall that $D(\Gamma)$ is a dense subset of $\mathcal H^{1/2}(\Gamma)$. Using \eqref{split} it is not hard to show $f = F_i|_{\Gamma} - F_e|_{\Gamma}$ by an approximation process. The details are left to the reader.

\section{Fractional Sobolev spaces on quasidisks}
 Let $\Gamma $ be a quasicircle and $\Omega_{i,e}$ its complementary domains as above. In Section $1$ we have defined $\mathcal{H}^s(\Gamma)$: for $s>1/2$ it is the set of $(s-1/2)$-H\"older functions whose harmonic extension to $\Omega_{i,e}$ belongs to $\mathcal{H}^s(\Omega_{i,e})$ assigned the norm $\|\cdot\|_{\mathcal H^s(\Gamma)}$  and for $0<s<1/2$ (also $s = 1/2$), the completion of $D(\Gamma)$ under the norm $\|\cdot\|_{\mathcal H^s(\Gamma)}$, the square of the norm being for both cases
$$
\iint_{\Omega_i}|\nabla u_i(z)|^2 d(z,\Gamma)^{1-2s} dxdy+\iint_{\Omega_e}|\nabla u_e(z)|^2d(z,\Gamma)^{1-2s}dxdy.
$$

By the inclusion relation $\mathcal H^s(\Gamma) \subset \mathcal H^{1/2}(\Gamma)$ for $1/2 < s <1$, we conclude by the preceding section that  any $f \in \mathcal H^s(\Gamma)$  can be written as $f = F_i|_{\Gamma} - F_e|_{\Gamma}$ with $F_i, F_e$ (see \eqref{Fe2},\eqref{Fi2}) 
being holomorphic in $\Omega_{i,e}$ and $F_{i,e} \in \mathcal H^{1/2}(\Omega_{i,e})$.  
 We wish to find conditions on $\Gamma$ that would imply that $F_{i,e}\in \mathcal{H}^s(\Omega_{i,e})$ whenever $f \in \mathcal H^s(\Gamma)$, $0<s<1$. 
 
 For this purpose and  later use,  we need to introduce weights of Muckenhoupt $A_p$ on $\mathbb C$. A 
 locally integrable nonnegative functions $\omega:\mathbb C\to \mathbb R$ 
 is said to satisfy the $A_1$ condition on $\mathbb C$ if there exists $C>0$ such that for any disk $D$ of the plane,   
$$\frac{1}{|D|}\iint_D\omega(z)dxdy\leq C\omega(z)$$ 
for almost all $z\in D$; while $\omega$ is said to satisfy the $A_{\infty}$ condition if
$$
\frac{1}{|D|}\iint_D\omega(z)dxdy\leq C\exp \left(\frac{1}{|D|}\iint_D\log\omega(z)dxdy\right). 
$$ 
For any $p>1$, the $A_p$ weight  is the set of locally integrable nonnegative functions $\omega:\mathbb C\to \mathbb R$ such that there exists $C>1$ such that for every disk $D$ of the plane,
$$ \frac{1}{|D|}\iint_D\omega(z)dxdy\left(\frac{1}{|D|}\iint_Dw(z)^{-\frac{1}{p-1}}dxdy\right)^{p-1}\leq C.$$
The $A_p$ weight on $\mathbb T$ can be defined similarly. 
 It is easy to see that $A_1\subset A_p \subset A_{\infty}$ for all $p>1$, and $A_{p_1}\subset A_{p_2}$ for $1 < p_1 < p_2$. We can then assume $p > 2$ whenever the weight $\omega$ has $A_p$ for $p \geq 1$, and by H\"older inequality the left part of the above inequality is always $\ge 1$: one can thus interpret this class as verifying some kind of reverse H\"older inequality. If the weight $\omega$ has $A_p$ then any Calderon-Zygmund operator is bounded on the weighted space $L^p(\mathbb C, \omega(z)dxdy)$ (\cite{CoF}). 
 Noting that $F_i'(z) = B(\bar\partial u_e \chi_{\Omega_e})(z)$,  $F_e'(z) = -B(\bar\partial u_i \chi_{\Omega_i})(z)$ and the Beurling operator $B$  is a Calder\'on-Zygmund operator, we have that a sufficient condition on $\Gamma$ that implies $F_{i,e}\in \mathcal{H}^s(\Omega_{i,e})$  is that the weight $\omega(z)=d(z,\Gamma)^{1-2s}$ satisfies the $A_2$ condition. 
Precisely, under this condition we have
\begin{align*}
    &\|F_i\|_{\mathcal H^s(\Omega_i)} \leq C \|u_e\|_{\mathcal H^s(\Omega_e)} \leq C \|f\|_{\mathcal H^s(\Gamma)};\\
    &\|F_e\|_{\mathcal H^s(\Omega_e)} \leq C \|u_i\|_{\mathcal H^s(\Omega_i)} \leq C \|f\|_{\mathcal H^s(\Gamma)},
\end{align*}
where $C$ is an absolute constant.
 
The dependence of the condition that $d(z,\Gamma)^{1-2s}$ having $A_2$ on the geometry of the curve has been studied by Astala \cite{Ast}. In order to understand his results we first need to introduce some notions.

For a compact set $E\subset \mathbb C$ and $0<\delta\le 2$ let
$$ M_\delta(E,t)=\frac{|E+B(0,t)|}{t^{2-\delta}}.$$
We then define a kind of Minkowski content by
$$ h_\delta(E)=\sup\limits_{0<t\leq \mathrm{diam}(E)}{M_\delta(E,t)}.$$
\begin{df} We say that a Jordan curve $\Gamma$ is $\delta$-regular if there exists $C>0$ such that for every disk $D(z,R)\subset \mathbb C$,
$$h_\delta(\Gamma\cap D(z,R))\leq CR^\delta.$$
\end{df}
\noindent It is known (\cite{Ast}) that $1$-regularity is equivalent to Ahlfors-David regularity.

Astala \cite{Ast} has proven that for any quasicircle $\Gamma$ there exists $\delta<2$ such that $\Gamma$ is $\delta$-regular. We may thus have, for quasicrcles,
$$h(\Gamma)=\inf\{\delta: \Gamma\, \mathrm{is}\, \delta\text{-}\mathrm{regular}\}<2.$$
 
\begin{df} A compact subset $E$ of the complex plane is said to be porous if there exists $c\in (0,1)$ and $r_0 > 1$ such that for every $z\in \mathbb C$ and $0 < r \leq r_0$, the disk $D(z,r)$ contains a disk of radius $cr$ which does not intersect $E$.
\end{df}
 
\begin{theo}[\cite{Ast}]\label{As1} 
Let $\alpha\in (0,1)$ and $\Gamma$ be a porous Jordan curve. We  have
\begin{equation}
    d(z,\Gamma)^{\alpha-1}\in A_2 \Leftrightarrow d(z,\Gamma)^{\alpha-1}\in A_1 \Leftrightarrow\alpha >h(\Gamma)-1.
\end{equation} 
\end{theo}
Notice that every quasicircle is porous (see \cite{Va}). 
We may now state the principal theorem of this section:
\begin{theo} \label{As2} 
Let $0<s<1$ and  $\Gamma$ be a quasicircle. If the point $(h(\Gamma),s)$ locates in the shadow (see Figure \ref{domain}); that is, 
$$\frac{h(\Gamma)-1}{2}<s<\frac{3-h(\Gamma)}{2},$$  
then the following two statements hold:
\begin{enumerate}
\item[\rm(1)] any $f\in \mathcal{H}^s(\Gamma)$ can be written as $f=F_i|_{\Gamma}-F_e|_{\Gamma}$ with $F_i$ and $F_e$ being analytic in $\Omega_i$ and $\Omega_e$, respectively; 
\item[\rm(2)]  Moreover, $F_{i,e}\in \mathcal{H}^s(\Omega_{i,e})$ and  exists $C>0$ such that
$\|F_{i,e}\|_{\mathcal{H}^s(\Omega_{i,e})}\leq C \|f\|_{\mathcal{H}^s(\Gamma)}$.
\end{enumerate}
\end{theo}

\begin{proof}
Following the discussion above, it is clear that  statement $(2)$ holds if the weight $d(z,\Gamma)^{1-2s}\in A_2$. 
Assume first that $s\in(0,1/2)$: then, by the definition of $A_2$, $d(z,\Gamma)^{1-2s}\in A_2 \Leftrightarrow d(z,\Gamma)^{2s-1}\in A_2$, which is equivalent, by Theorem \ref{As1}, to the fact that $s>\frac{h(\Gamma)-1}{2}$. 
Assume now $s\in (1/2,1)$: then, again by Theorem \ref{As1}, $d(z,\Gamma)^{1-2s}\in A_2 \Leftrightarrow s<\frac{3-h(\Gamma)}{2}$. The case of $s = 1/2$ has been addressed in the last section.

Assume $\frac{1}{2} \leq s < \frac{3-h(\Gamma)}{2}$. Then  statement $(1)$ holds as we have already noticed. Assume now that $\frac{h(\Gamma) - 1}{2} < s < \frac{1}{2}$. By the preceding section statement $(1)$ holds whenever $f \in D(\Gamma) \subset \mathcal H^s(\Gamma)$.  Then, using an approximation process we can show that  statement $(1)$ holds for a general $f \in \mathcal H^s(\Gamma)$. 
\end{proof}

 \begin{figure}
    \centering
    \includegraphics[width=.6\textwidth]{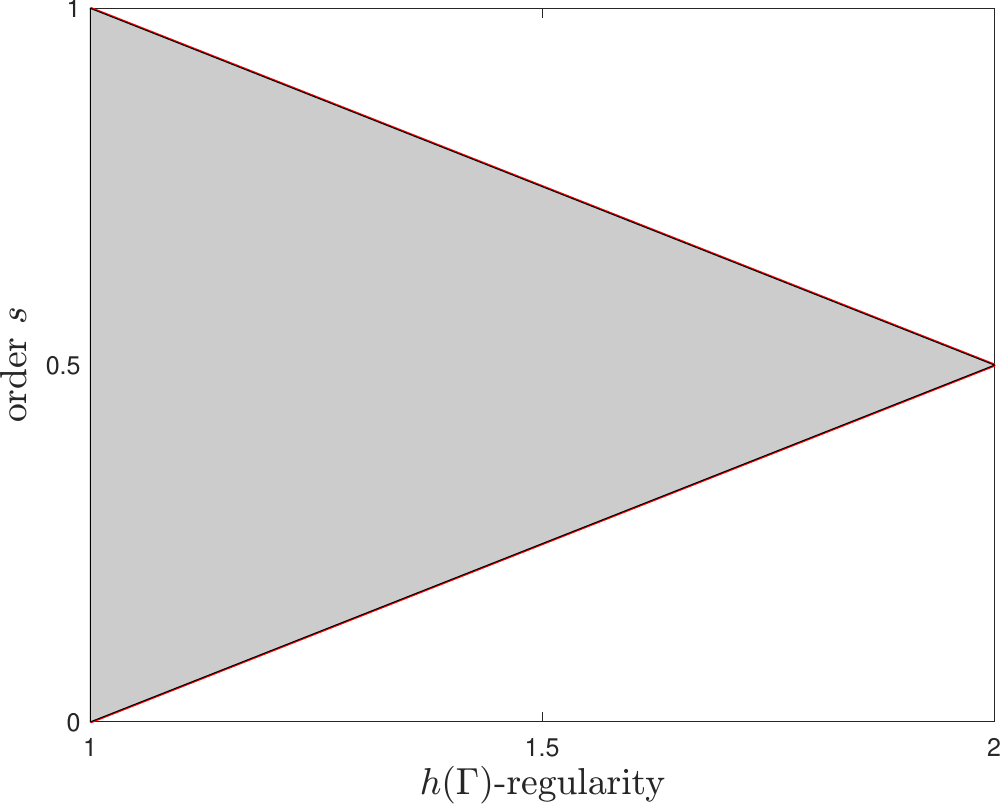}
    \caption{Domain formed by points $(h(\Gamma), s)$}
    \label{domain}
\end{figure}

We end up this section with a uniqueness assertion. 
Let $\Gamma$ be a quasicircle in $\mathbb C$. Liu-Shen \cite{LS} showed that the Calder\'on-Plemelj problem has a unique solution in $\mathcal H^{1/2}(\Gamma)$. By the inclusion relation 
$\mathcal H^s(\Gamma)\subset\mathcal H^{1/2}(\Gamma)$ for $1/2 < s <1$ it is easily seen that  the Calder\'on-Plemelj problem has a unique solution in $\mathcal H^{s}(\Gamma)$.

\section{Fractional Sobolev spaces on Chord-arc domains}
We start this section from the classical Sobolev space on $\mathbb R^n$ (see \cite{AS}). The Sobolev space $H^{\alpha}(\mathbb R^n)$ of non-negative integer order $\alpha$ is defined as the set of all functions $u \in L^2(\mathbb R^n)$ such that for every multi-index $\alpha$ with $|i| = \alpha$, the mixed partial derivative exists in the distributional sense and is in $L^2(\mathbb R^n)$, namely, 
$$
\|u\|_{H^{\alpha}(\mathbb R^n)}^2 = \sum_{|i| = \alpha}\int_{\mathbb R^n} |D_iu|^2 dx < \infty.
$$
The right summation  has an equivalent expression in terms of the Fourier transform as follows
$$
d_{\alpha}(u) =\int_{\mathbb R^n} |\xi|^{2\alpha}|\hat u(\xi)|^2d\xi;
$$
that can be used to define the Sobolev space $H^{\alpha}(\mathbb R^n)$ for any $\alpha \ge 0$. For the convenience of generalization to functions on a general open set of $\mathbb R^n$, we now introduce an expression for $d_{\alpha}(u)$ when $\alpha$ is a non-integer, which involves $u$ and its derivatives, but not the Fourier transform.

If $0 < \alpha <1$, it follows from Parseval's formula that 
$$
d_{\alpha}(u) = \frac{1}{C(n, \alpha)}\int_{\mathbb R^n}\int_{\mathbb R^n}\frac{|u(x)-u(y)|^2}{|x-y|^{n+2\alpha}}dxdy.
$$
Here, the constant 
$$
C(n, \alpha) = \frac{2^{-2\alpha+1}\pi^{\frac{n+2}{2}}}{\Gamma(\alpha + 1)\Gamma(\alpha + \frac{n}{2})\sin(\pi\alpha)}. 
$$
Then we define the Sobolev space of fractional order $\alpha \in (0, 1)$ as being the set of measurable functions $u \in L^2(\mathbb R^n)$ such that
$$
\|u\|_{H^{\alpha}(\mathbb R^n)}^2 = \int_{\mathbb R^n}\int_{\mathbb R^n}\frac{|u(x)-u(y)|^2}{|x-y|^{n+2\alpha}}dxdy < \infty.
$$
If $\alpha = k + s$ where $k$ is a positive integer and $0 < s < 1$, we then have 
$$
d_{\alpha}(u) = \frac{1}{C(n, s)}\sum_{|i| = k}\|D_iu\|^2_{H^s(\mathbb R^n)}.
$$
In this case we may define the Sobolev space $H^{\alpha}(\mathbb R^n)$ as the set of functions $u \in L^2(\mathbb R^n)$ such that
$$
\|u\|_{H^{\alpha}(\mathbb R^n)}^2 = \|u\|_{H^{k}(\mathbb R^n)}^2 + \sum_{|i| = k}\|D_iu\|^2_{H^s(\mathbb R^n)} < \infty.
$$

Let $\Omega$ be a chord-arc domain of the plane, namely, its boundary curve $\Gamma$ is an image of the unit circle by a bi-Lipschitz homeomorphism of the plane. Here, $\Gamma$ is called the chord-arc curve. Let us recall a geometric description of the chord-arc curve: there exists $K \geq 1$ such that for any pair of points $(z_1, z_2)$ of $\Gamma$,  
the length of the shorter arc $\Gamma(z_1, z_2)$ of the two sub-arcs of $\Gamma$ with endpoints $z_1$, $z_2$ satisfies the estimate
$$
\text{length}(\Gamma(z_1, z_2)) \leq K|z_1 - z_2|.
$$
($\Gamma$ is a quasicircle if the above estimate holds with length replaced by diameter.) 
Noting that the chord-arc curve $\Gamma$ is rectifiable, we may then assume, without loss of generality, that its length is $2\pi$, so that if $z:[0,2\pi)\to \Gamma$ stands for the arc-length parametrization of $\Gamma$, the mapping $\lambda: \mathbb T\to \Gamma,\; e^{it}\mapsto z(t)$ is another bi-Lipschitz map with $|\lambda'| = 1$, and moreover, it is easy to show that 
$$
\frac{1}{K}\leq\frac{|\lambda(e^{it_1}) - \lambda(e^{it_2})|}{|e^{it_1} - e^{it_2}|}\leq \frac{\pi}{2}
$$
for any $t_1, t_2 \in [0, 2\pi)$.

Motivated by the above definition of Sobolev spaces on $\mathbb R^n$, 
we may define the Sobolev space $H^{\alpha}(\Omega)$, $1/2 < \alpha < 3/2$, as a set 
of functions $u \in L^2(\Omega)$ such that $\|u\|_{H^{\alpha}(\Omega)} < \infty$. The quotient space of $H^{\alpha}(\Omega)$ by the constant functions is a Hilbert space  with a natural norm $\|\cdot\|_{H^{\alpha}(\Omega)}$. On the other hand, we define the fractional Sobolev trace space $H^s(\Gamma)$, $0<s<1$,  as the space of  functions $f \in L^2(\Gamma)$ such that
$$\|f\|_{H^s(\Gamma)}^2 = \iint_{\Gamma\times\Gamma}\frac{|f(z)-f(\zeta)|^2}{|z-\zeta|^{1+2s}}d\sigma(z) d\sigma(\zeta)<\infty.$$
This is  a Hilbert space modulo constants with a norm $\|\cdot\|_{H^s(\Gamma)}$.

Notice that the norms defining $H^s(\Omega)$ and $H^s(\Gamma)$ are bi-Lipschitz invariant. We can thus conclude that the spaces $H^s(\Omega)$ are isomorphic to $H^s(\mathbb D)$, as well as $H^s(\Gamma)$ to $H^s(\mathbb T)$, the latter isomorphism being induced by the bi-Lipschitz map $\lambda$ so that $\|f\|_{H^s(\Gamma)}^2$ is equivalent to 
$$
 \iint_{\mathbb T\times \mathbb T}\frac{|f\circ\lambda(\zeta_1) - f\circ\lambda(\zeta_2)|^2}{|\zeta_1-\zeta_2|^{1+2s}}|d\zeta_1||d\zeta_2|. 
$$
We define $H^0(\Gamma)$ as the space of functions $f \in L^2(\Gamma, d\sigma)$, and equivalently, $f\circ\lambda \in L^2(\mathbb T)$ and $H^1(\Gamma)$ as the set of functions $f \in L^2(\Gamma, d\sigma)$ such that $f\circ\lambda$ is an anti-derivative of a function in $L^2(\mathbb T)$.

It is a classical fact that $C_c^\infty(\mathbb R^2)$ is dense in $H^{\alpha}(\mathbb R^2)$. 
We may define the trace operator  $\gamma$ on $C_c^\infty(\mathbb R^2)$ by $\gamma(u)=u|_\Gamma$; that belongs to $H^s(\Gamma)$, $s = \alpha - 1/2$, the order having $1/2$-loss.    Moreover, 
\begin{theo}[Page 182 in \cite{JW1}]
    Let $\alpha \in (1/2, 3/2)$. The trace operator $\gamma$ can be extended to a continuous map from $H^{\alpha}(\mathbb R^2)$ onto $H^s(\Gamma)$. 
\end{theo}
   The restriction map from $H^{\alpha}(\mathbb R^2)$ to $H^\alpha(\Omega)$ is well-defined, and it is continuous and surjective onto $H^\alpha(\Omega)$, see e.g. \cite{JW1, Jon}. The surjectivity guarantees the existence of an extension operator for $\Omega$; that makes $H^{\alpha}(\Omega)$ inherits many properties possessed by $H^{\alpha}(\mathbb R^2)$. In particular, $C_c^{\infty}(\mathbb R^2)|_{\Omega}$ is a dense subset  of $H^\alpha(\Omega)$. 
   Let $u \in H^{\alpha}(\Omega)$ and $\tilde u \in H^{\alpha}(\mathbb R^2)$ such that $\tilde u|_{\Omega} = u$. It can be shown that the operator $\gamma(u)$, defined as $\gamma(\tilde u)$, does not depend on the choice of the extension $\tilde u \in H^{\alpha}(\mathbb R^2)$ (see Theorem 8.7 in \cite{BMMM} and also \cite{VW}, \cite{Din}). 
   We then get to the following theorem. However, we will not use it in our proof.  
   \begin{theo}
       The trace operator $\gamma$ is a continuous map from $H^{\alpha}(\Omega)$ to $H^s(\Gamma)$, whose kernel is  $H_0^\alpha(\Omega)$, the closure of $C_c^\infty(\Omega)$ in $H^\alpha(\Omega)$.
   \end{theo}

We now  come back  to the Plemelj-Calder\'on problem. By a deep theorem of David \cite{Dav} using the Lipschitz result \cite{CMM}, the Cauchy integral operator $T$ is bounded on $L^2(\Gamma, d\sigma)$ (i.e., $H^0(\Gamma)$). Using it we will show the boundedness of the operator $T$ on $H^1(\Gamma)$.
\begin{theo}\label{H1}
    The operator $T$ is bounded on the Sobolev space $H^1(\Gamma)$, and moreover, the operator norm satisfies that 
    $$\|T\|_{H^1(\Gamma)\to H^1(\Gamma)} = \|T\|_{L^2(\Gamma)\to L^2(\Gamma)}.$$
\end{theo}

\begin{proof}
Let  $f\in H^1(\Gamma)$ so that $f\circ \lambda$ is an anti-derivative of a function $h\in L^2(\mathbb T)$. Let us define the function $f'$ 
by $f'\circ\lambda \lambda' = h$. Then $f'\in L^2(\Gamma, d\sigma)$. The space $H^1(\Gamma)$ is endowed with the natural norm $\|f\|_{H^1(\Gamma)} = \|f'\|_{L^2(\Gamma, d\sigma)}$. 
By David's theorem we may write 
\begin{equation}\label{decom}
    f'=\varphi_i + \varphi_e
\end{equation}
on $\Gamma$. Here, $\varphi_i$ and $\varphi_e$ belong to the Hardy spaces of index $2$ on $\Omega_i$ and $\Omega_e$, respectively, and then $\varphi_i$ and $\varphi_e$ have non-tangential boundary limits almost everywhere on $\Gamma$ with respect to the arc-length measure, denoted still by $\varphi_i$ and $\varphi_e$, so that $\|\varphi_i\|_{L^2(\Gamma, d\sigma)}$ and $\|\varphi_e\|_{L^2(\Gamma, d\sigma)}$ are both controlled from above by $\|f\|_{H^1(\Gamma)}$.

 Set $\Gamma_r = \tau(|z| = r)$ to be  ``circular curves'', where $\tau$ is a Riemann map that takes $\mathbb D$ onto $\Omega$. 
By the definition, the Hardy space  $E^p(\Omega)$ of index $p$, $p\ge 1$, on the chord-arc domain $\Omega$ consists of holomorphic functions $F$ on $\Omega$ with a finite norm
$$
\Vert F \Vert_p =\Bigg (\frac{1}{2\pi}\sup_{r}\int_{\Gamma_r}|F(w)|^p d\sigma(w) \Bigg)^{1/p}.
$$
 Notice that $\tau$ is a homeomorphism of the closures $\bar{\mathbb D}$ onto $\Omega\cup\Gamma$, and absolutely continuous on $\mathbb T$ since $\Gamma$ is rectifiable.  
 
  Recall that $\varphi_i \in E^2(\Omega_i)$; that is included in $E^1(\Omega_i)$, we then have $\varphi_i\circ\tau\tau' \in E^1(\mathbb D_i)$. Let now $\Phi_i$ be an anti-derivative of $\varphi_i$ on $\Omega_i$. Then we see $(\Phi_i\circ\tau)' = \varphi_i\circ\tau\tau' \in E^1(\mathbb D_i)$. By a result of Hardy-Littlewood (see e.g. Page 89 in \cite{Gar}), we conclude that $\Phi_i\circ\tau$ is continuous on $\bar{\mathbb D_i}$ and absolutely continuous on $\mathbb T$ such that 
$(\Phi_i\circ\tau)'(\zeta) = \lim_{r \to 1}(\Phi_i\circ\tau)'(r\zeta)$ almost everywhere on $\mathbb T$, and thus $(\Phi_i\circ\tau)'(\zeta) = \varphi_i\circ\tau(\zeta)\tau'(\zeta)$ almost everywhere on $\mathbb T$. Using it, let us define $\Phi_i' = \varphi_i$ on $\Gamma$, and we can similarly define  $\Phi_e' = \varphi_e$ on $\Gamma$. 
Combined with \eqref{decom}  that leads to, by adjusting the constants,  $$f=\Phi_i+\Phi_e$$
on $\Gamma$ 
with the norm of $\Phi_{i,e}$ in $H^1(\Gamma)$ controlled from above by $\|f\|_{H^1(\Gamma)}$. By  Plemelj's formula we conclude that the operator $T$ is bounded on $H^1(\Gamma)$ with respect to the norm $\|\cdot\|_{H^1(\Gamma)}$, and moreover, we have $\|T\|_{H^1(\Gamma)\to H^1(\Gamma)} = \|T\|_{L^2(\Gamma)\to L^2(\Gamma)}$. 
\end{proof}

Now we get ready to show the main result of this section: the operator $T$ is bounded on the fractional Sobolev space $H^s(\Gamma)$ for $0 < s < 1$. 
Because of  the isomorphism between $H^s(\Gamma)$ and $H^s(\mathbb T)$,  this boils down to proving that the operator with kernel
$$ \text{P.V.}\frac{\lambda'(\zeta)}{\lambda(\zeta)-\lambda(\xi)}$$
is bounded on $H^s(\mathbb T)$. Set $g = f\circ\lambda$. 
This operator is more precisely defined by
$$\widetilde Tg(\xi)=\frac{1}{2\pi i}\text{P.V.}\int_{\mathbb T}\frac{g(\zeta)\lambda'(\zeta)}{\lambda(\zeta)-\lambda(\xi)}d\zeta.$$

Let us introduce Calder\'on's interpolation theorem.
\begin{theo}[Page 38 in \cite{Tri}, Chapter 6 in \cite{BL} and also \cite{Cal1}]\label{Cald}Let $s_0, s_1 \in [0, 1]$. If $s_0 \ne s_1$ then we have, in terms of complex interpolation theory of Banach spaces,
$$ H^s(\mathbb T)=[H^{s_0}(\mathbb T),H^{s_1}(\mathbb T)]_{\theta}$$
where  $s = (1-\theta)s_0 +\theta s_1$ and the exponent $\theta \in (0, 1)$. 
\end{theo}

\noindent The following functorial property of complex interpolation  is a basic assertion in interpolation theory. 
Let $\{A_0, A_1\}$ and $\{B_0, B_1\}$  be two interpolation couples of Banach spaces and let $L$ be a linear operator mapping from $A_0+A_1$ into $B_0+B_1$ such that its restriction to $A_j$ is a linear and bounded operator from $A_j$ into $B_j$ with norm $M_j$, where $j = 0, 1$. Then the restriction of $L$ to 
$[A_0, A_1]_{\theta}$, $0 < \theta < 1$ is a linear and bounded operator from $[A_0, A_1]_{\theta}$ into $[B_0, B_1]_{\theta}$ with norm $M_\theta$. Further, it is known that the complex interpolation method is an exact interpolation functor of exponent $\theta$ (see Page 88 in \cite{BL}) meaning that 
\begin{equation}\label{norm}
    M_{\theta} \leq M_0^{1-\theta}M_1^{\theta}.
\end{equation}

It now suffices to invoke Calder\'on's interpolation theorem to conclude the following. 
\begin{theo}
    For $0 < s < 1$, the operator $T$ is bounded on $H^s(\Gamma)$.
\end{theo}
\begin{proof}
By taking $s_0 = 0$ and $s_1 = 1$ in Theorem \ref{Cald}, we have that
$$
H^s(\mathbb T) = [L^2(\mathbb T), H^1(\mathbb T)]_s.
$$
By David's theorem and Theorem \ref{H1}, we conclude that the operator $\widetilde T$ is bounded on $H^s(\mathbb T)$ with the operator norm $\|\widetilde T\|_{H^s(\mathbb T)\to H^s(\mathbb T)} = \|\widetilde T\|_{L^2(\mathbb T)\to L^2(\mathbb T)}$, and equivalently, the operator $T$ is bounded on $H^s(\Gamma)$ with the operator norm $\|T\|_{H^s(\Gamma)\to H^s(\Gamma)} = \|T\|_{L^2(\Gamma)\to L^2(\Gamma)}$.
    \end{proof}
 It follows from  Plemelj's  formula that every function $f\in H^s(\Gamma)$ may be written uniquely as $f=\Phi_i+\Phi_e$ with $\|\Phi_{i,e}\|_{H^s(\Gamma)}\leq C\|f\|_{H^s(\Gamma)}$ for some constant $C$, $\Phi_{i,e}$ being boundary values of holomorphic functions in $\Omega_i$ and $\Omega_e$, respectively. The uniqueness of decomposition  for $0 < s \leq 1$  just follows from the case of $s = 0$.

A theorem by Murai \cite{Mur} states that for a Lipschitz curve $\Gamma$ with Lipschitz norm $M$, $\|T\|_{L^2(\Gamma)\to L^2(\Gamma)}$ is no more than $C(1+M)^{3/2}$ (see also \cite{Dav87}), where $C$ is a universal constant. Plugging this information in the preceding theorem we obtain the same bound for $\|T\|_{H^{1/2}(\Gamma)\to H^{1/2}(\Gamma)}$. But the norm  $\|T\|_{H^{1/2}(\Gamma)\to H^{1/2}(\Gamma)}$ depends only on the $L^2$ boundedness of the Beurling transform (see Section 2), so that it is actually independent of $M$. 
In order to get better estimates of $\|T\|_{H^{s}(\Gamma)\to H^{s}(\Gamma)}$ we thus use Calder\'on's interpolation theorem between $L^2(\Gamma)$ and $H^{1/2}(\Gamma)$ for $0<s<1/2$ and between $H^{1/2}(\Gamma)$ and $H^1(\Gamma)$ for $1/2<s<1$. We use more precisely that 
\begin{align*}
    H^s(\Gamma)&=[L^2(\Gamma),H^{1/2}(\Gamma)]_{2s},\;\;\;\;\;\,0<s<1/2;\\
    H^s(\Gamma)&=[H^{1/2}(\Gamma),H^1(\Gamma)]_{2s-1},\;\;1/2<s<1.
\end{align*}
Using \eqref{norm}, 
we get to 
\begin{theo} If $\Gamma $ is a Lipschitz curve with Lipschitz constant $M$ then we have, for $0<s<1$, 
$$
\|T\|_{H^{s}(\Gamma)\to H^{s}(\Gamma)} \leq C(1+M)^{\frac 32|1-2s|}
$$
where $C$ is a universal constant. 
\end{theo}

\section{Douglas versus Littlewood-Paley: the chord-arc case}
\subsection{The conjugate operator}
Let $\Omega$ be a Jordan domain whose boundary $\Gamma$ is assumed to be a chord-arc curve of length $2\pi$, and $0<s<1$. We have seen two ways of generalizing  the fractional Sobolev spaces from $\mathbb R$ (or from $\mathbb T$) to $\Gamma$. The first one, which can be called the Douglas way is
$$ H^s(\Gamma)=\bigg\{ f\in L^2(\Gamma): \iint_{\Gamma\times\Gamma}\frac{|f(z)-f(\zeta)|^2}{|z-\zeta|^{1+2s}}d\sigma(z) d\sigma(\zeta) < \infty \bigg\}.$$
The second one, via Littlewood-Paley theory, is
$$ \mathcal{H}^s(\Gamma)=\bigg\{ f\in L^2(\Gamma):\iint_{\Omega_{i,e}} |\nabla u_{i,e}(z)|^2 d(z,\Gamma)^{1-2s}  dxdy<\infty\bigg\}.$$ 
Here,  $u_{i,e}$ stands for the harmonic extension of $f$ to $\Omega_{i,e}$.
We have seen that the Plemelj-Calder\'on problem is solvable for these values of $s$ on both spaces $H^s(\Gamma)$ and $\mathcal{H}^s(\Gamma)$ so that the natural question arises of whether these spaces coincide.

Let $z_0$ be a point in $\Omega$. If $u$ is harmonic in $\Omega$ it is well known that there exists a unique harmonic function $\tilde{u}$ on $\Omega$ such that $\tilde{u}(z_0)=0$ and  $u+i\tilde{u}$ is holomorphic in $\Omega$. In the case of $\Omega=\mathbb{D}$, $\tilde{u}$ is the harmonic (Poisson) extension of $H(f)$ where $H$ is the Hilbert transform:
 $$
    H(f)(e^{i\theta}) = \lim_{\epsilon \to 0}\frac{1}{2\pi}\int_{|\theta-\varphi|>\epsilon} \cot\Big(\frac{\theta-\varphi}{2}\Big)f(e^{i\varphi})d\varphi.
    $$
By Cauchy-Riemann equations, $|\nabla\tilde{u}|=|\nabla u|$, so that $\mathcal{H}^s(\Omega)$ is conjugate-invariant. A necessary condition for $H^s(\Gamma)=\mathcal{H}^s(\Gamma)$ to hold is thus that $H^s(\Gamma)$ is stable by conjugation. Let $\varphi:\,\mathbb D\to \Omega$ be the Riemann map with $\varphi(0)=z_0$. Recall that $t\mapsto z(t)$ denotes the arc-length parametrization of $\Gamma$ and $\lambda(e^{it}) = z(t)$ so that $\varphi=\lambda\circ h$ where $h$ is an absolutely continuous homeomorphism of $\mathbb T$ with $|h'| = |\varphi'|$. A simple computation shows that if $f\in H^s(\Gamma)$; that is $g = f\circ\lambda \in H^s(\mathbb T)$ and $u$ is its harmonic extension to $\Omega$, then $\tilde{u}$ is the harmonic extension of $\tilde f = \tilde g\circ\lambda^{-1}$. Here, 
$$ \tilde{g}=V_h^{-1}\circ H\circ V_h(g),$$
where $V_h(g) = g\circ h$. 

Before we state a theorem about the general case $0 \leq s \leq 1$, let us briefly comment on the limiting case $s=0$; that is the case of $L^2(\mathbb T)$. In this case it is known (see \cite{CoF},   Page 247 in \cite{Gar}) that $V_{h}^{-1}HV_h$ is bounded on $L^2(\mathbb T)$ if and only if $|h'| = |\varphi'|$ belongs to the weight of Muckenhoupt  $A_2$ on  $\mathbb T$.  
It is known (see \cite{JK}) that $\Gamma$ being chord-arc implies $|\varphi'|$ having
$A_\infty$ but there are examples of chord-arc curves such that $|\varphi'|\notin A_2$ (see \cite{JoZi}). A natural condition on the  curve $\Gamma$ implying that $|\varphi'|\in A_2$ is that it is the graph in polar coordinates of a Lipschitz function $\theta\mapsto r(\theta)$.  
In this case, it can be shown that $|\varphi'|$ satisfies the Helson-Szeg\"o condition: $\log |\varphi'| = u + Hv$ with $u\in L^{\infty}$ and $\|v\|_{\infty} < \pi/2$, which is equivalent to $|\varphi'| \in A_2$ as follows from the Helson-Szeg\"o theorem (\cite{Gar}). We may now state:
\begin{theo}\label{conj}
If $\Gamma$ is such that $|\varphi'|\in A_2$ on $\mathbb T$ then, for $0\le s\le 1$, $H^s(\Gamma)$ is stable by conjugation, and moreover, the conjugate operator $f\mapsto\tilde f$ on $H^s(\Gamma)$ is bounded. 
\end{theo}
\begin{proof}
    We have just seen that $V_h^{-1}HV_h$ is bounded on $L^2(\mathbb T)$ if and only if $|\varphi'|\in A_2$.  
Let now $f\in H^1(\Gamma)$ so that $g = f\circ\lambda \in H^1(\mathbb T)$ as before. 
We have $(g\circ h)'=g'\circ h h'\in L^2(1/|h'|)$. Since $|h'|\in A_2$, and the same is true for $1/|h'|$, we have $H((g\circ h)')\in L^2(1/|h'|)$ with a norm bounded by the $L^2(1/|h'|)$-norm of $(g\circ h)'$ 
(see \cite{CoF}); that implies $(\tilde g)'\in L^2(\mathbb T)$ with $\tilde g = V_{h}^{-1}HV_h(g)$, and thus $\tilde f = \tilde g \circ \lambda^{-1} \in H^1(\Gamma)$ such that $\|\tilde f\|_{H^1(\Gamma)}$ bounded by $\|f\|_{H^1(\Gamma)}$. 
The theorem now follows from Calder\'on's interpolation theorem.
\end{proof}

\subsection{The operator $V_s$}
Recall that $\mathcal H^s(\Omega)$, $0\leq s<1$, is the set of harmonic functions $u$ on the domain $\Omega$ bounded by the Jordan curve $\Gamma$ such that its norm
$\|u\|_{\mathcal H^s(\Omega)} < \infty$. Suppose $\Gamma$ is chord-arc  which in particular is rectifiable. It is known that the non-tangential boundary values $f$ of $u$ exist almost everywhere on $\Gamma$ with respect to the arc-length measure and $u$ can be recovered from $f$ using the harmonic extension.  The space of boundary functions of $\mathcal H^s(\Omega)$ is denoted by $\mathcal H^s(\Omega\!\to\!\Gamma)$ so that $\mathcal H^s(\Gamma)=\mathcal H^s(\Omega_i\!\to\!\Gamma)\cap \mathcal H^s(\Omega_e\!\to\!\Gamma)$.  
We may identify $\mathcal H^s(\Omega)$ with its boundary traces  for convenience, and will switch between the two freely. 
We denote the subspace of $\mathcal H^s(\Omega)$ consisting of analytic functions by $\mathcal A^s(\Omega)$, i.e., the set of $u+i\tilde u$ with $u \in \mathcal H^s(\Omega)$.

Let $\Omega,\,\Omega'$ be two Jordan domains  containing $0$ and $\varphi$ a holomorphic diffeomorphism from $\Omega$ onto $\Omega'$ fixing $0$. Let $f\in \mathcal{A}^s(\Omega')$. Using the change of variable $\zeta=\varphi(z)$ we have
$$\iint_{\Omega'} d(\zeta,\Gamma')^{1-2s}|f'(\zeta)|^2d\xi d\eta=\iint_\Omega d(\varphi(z),\Gamma')^{1-2s}|(f\circ\varphi)'(z)|^2dxdy.$$
By the Koebe distortion theorem there exists a universal constant $C>1$ such that
$$ C^{-1}d(z,\Gamma)|\varphi'(z)|\le d(\varphi(z),\Gamma') \le Cd(z,\Gamma)|\varphi'(z)|,$$
so that 
\begin{align}\label{Vs}
C^{-|1-2s|}\iint_\Omega d(z,\Gamma)^{1-2s}|V_s(f)'(z)|^2dxdy 
&\le\iint_{\Omega'} d(\zeta,\Gamma')^{1-2s}|f'(\zeta)|^2d\xi d\eta\\
&\le C^{|1-2s|}\iint_\Omega d(z,\Gamma)^{1-2s}|V_s(f)'(z)|^2dxdy,\nonumber
\end{align}
where $V_s$ is the operator defined by
$$V_s(f)(z)=\int_0^z(f\circ\varphi)'(u)\varphi'(u)^{1/2-s}du.$$
In other words, the operator $V_s$ is a bounded isomorphism between $\mathcal{H}^s(\Omega)$
and $\mathcal{H}^s(\Omega')$ with the operator norm 
$$\|V_s\|,\,\|V_s^{-1}\|\le C^{|1/2-s|}.$$
Notice that for $s=1/2$, this operator is nothing but the composition by $\varphi$ and that, in this case, $V_{1/2}$ is an isometry.

In order to understand better the operator $V_s$ let us rewrite it by using an integration by parts:
\begin{align*}
    V_s(f)(z) &=f\circ\varphi(z)\varphi'(z)^{1/2-s}-(1/2-s)\int_0^zf\circ\varphi(u)\varphi'(u)^{1/2-s}\frac{\varphi''(u)}{\varphi'(u)}du\\
    &=T_sf(z)-(1/2-s)S(T_s(f))(z),
\end{align*}
where 
$$T_sf(z)=f\circ\varphi(z)\varphi'(z)^{1/2-s}$$
and 
$$ Sg(z)=\int_0^z g(u)\frac{\varphi''(u)}{\varphi'(u)}du.$$

Let us now specialize to $s=0$ and $(\Omega,\Omega')=(\mathbb D,\Omega)$. Recall that the Hardy space  $E^2(\Omega)$ is the space of holomorphic functions $f:\,\Omega\to \mathbb C$ such that $T_0(f)\in E^2(\mathbb D)$, the classical Hardy space of the unit disk. We say that a function $g \in {\rm BMOA}(\mathbb D)$ if $g \in E^2(\mathbb D)$ and if in addition the boundary values of $g$ on $\mathbb T$ is of bounded mean oscillation (abbr. BMO) in the sense that 
$$
\sup_{I\subset\mathbb T}\frac{1}{|I|}\int_I |g(z) - g_I||dz| < \infty,
$$
where the supremum is taken over all sub-arcs $I$ of $\mathbb T$ and $g_I$ denotes the integral mean of $g$ over $I$. It is known that $\log\varphi' \in {\rm BMOA}(\mathbb D)$ if $\varphi(\mathbb D)$ is chord-arc, but not vice versa. By Fefferman-Stein, $\log\varphi' \in {\rm BMOA}(\mathbb D)$ if and only if 
$$d\mu =  \bigg|\frac{\varphi''(z)}{\varphi'(z)}\bigg|^2(1-|z|)dxdy$$
is a Carleson measure in $\mathbb D$ (see e.g. Chapter VI in \cite{Gar}).

\begin{theo}\label{E2A0}
Let $\Omega$ be a Jordan domain containing $0$ and $\varphi$ the Riemann mapping from $\mathbb D$ onto $\Omega$ fixing $0$. The following statements hold: 
\begin{enumerate}
    \item[\rm(1)] If $\log\varphi'\in {\rm BMOA}(\mathbb D)$ then  $ E^2(\Omega)\subset \mathcal{A}^0(\Omega)$;
    \item[\rm(2)] If $\Omega$ is a chord-arc domain then $\mathcal{A}^0(\Omega)\subset E^2(\Omega)$.
\end{enumerate}
Moreover, the inclusions are continuous with respect to the norms $\Vert\cdot\Vert_{E^2(\Omega)}$ and $\Vert\cdot\Vert_{\mathcal H^0(\Omega)}$. 
\end{theo}
\begin{proof}
    Suppose  $\log\varphi'\in {\rm BMOA}(\mathbb D)$; that is, $\mu$ is a Carleson measure in $\mathbb D$. 
 Let $f \in E^2(\Omega)$ so that $T_0f \in E^2(\mathbb D)$. Then by Carleson (see e.g. Theorem 3.9 in \cite{Gar}) 
$$
\iint_{\mathbb D}|T_0(f)|^2 d\mu \leq C\|T_0f\|_{E^2(\mathbb D)}^2
$$
where $C$ is a constant depending only on the Carleson norm of $\mu$. 
 Since $E^2(\mathbb D) = \mathcal A^0(\mathbb D)$ we have
$$
\iint_{\mathbb D}|(T_0f)'|^2(1-|z|)dxdy < \infty.
$$
Finally, 
\begin{align*}
    \iint_{\mathbb D} |(V_0f)'(z)|^2(1-|z|)dxdy &= \iint_{\mathbb D} \Big|(T_0f)'(z) - \frac{1}{2}T_0f(z)\cdot\frac{\varphi''(z)}{\varphi'(z)}\Big|^2(1-|z|)dxdy\\
    &\leq 2\iint_{\mathbb D}|(T_0f)'|^2(1-|z|)dxdy + \frac{1}{2}\iint_{\mathbb D}|T_0(f)|^2 d\mu\\
    & <\infty.
\end{align*}
Combined with \eqref{Vs}, 
this completes the proof of  statement $(1)$.

For the proof of $(2)$ we will need the definition of the (Littlewood-Paley) $\mathsf g$-function of a holomorphic function $f$ in $\mathbb D$:
$$ \mathsf g(f)(e^{i\theta})=\left(\int_0^1(1-r)|f'(re^{i\theta})|^2dr\right)^{1/2}.$$
Suppose now that $f\in \mathcal{A}^0(\Omega)$. Then $V_0(f)\in \mathcal{A}^0(\mathbb D)$, which in turn implies that
$$ \iint_{\mathbb D}(1-|u|)|(f\circ \varphi)'(u)|^2|\varphi'(u)|dudv < \infty,$$
or, in other words, that $\mathsf g(f\circ \varphi)\in L^2(\mathbb T,|\varphi'|d\theta)$. Since $\Gamma$ is assumed to be chord-arc, we have that $|\varphi'|$ has the weight $A_\infty$ on $\mathbb T$. 
We can then apply a theorem of Gundy and Wheeden \cite{GW} (see also \cite{JK}) which implies that $\mathsf g(f\circ \varphi)\in L^2(\mathbb T,|\varphi'|d\theta)$ if and only if  the non-tangential maximal function $\mathsf n(f\circ \varphi)$ of $f\circ\varphi$ is in $L^2(\mathbb T,|\varphi'|d\theta)$. On the other hand,  for any ``circular curve'' $\Gamma_r$  it holds that (see Page 233 in \cite{JK})
$$
\int_{\Gamma_r}|f(\zeta)|^2d\sigma(\zeta) \leq C\|\mathsf n(f\circ\varphi)\|^2_{L^2(\mathbb T,|\varphi'|d\theta)},
$$
which implies  $f\in E^2(\Omega)$. 
\end{proof}

Notice that for this theorem we do not need  $|\varphi'|$ having $A_2$ on $\mathbb T$.
This condition is nevertheless necessary for the following corollary to hold since if it is not attached then the (real) space of real parts of $E^2(\Omega\!\to\!\Gamma)$ functions is a proper subspace of the real space $L_\mathbb R^2(\Gamma, d\sigma)$  
\begin{cor}\label{0trace} 
Let $\Omega$ be a chord-arc domain bounded by $\Gamma$.  If $|\varphi'|\in A_2$ then $L^2(\Gamma, d\sigma)=\mathcal{H}^0(\Omega\!\to\!\Gamma)$.
\end{cor}

\subsection{Interpolation}
The aim of this last sub-section is to prove the following theorem using the interpolation of Bergman spaces and Calderon's interpolation theorem. 

\begin{theo} \label{interpol} 
Let $\Omega$ be a chord-arc domain bounded by $\Gamma$ and $\varphi$ its Riemann mapping fixing $0$. If $\Gamma$ is such that $|\varphi'| \in A_2$ on $\mathbb T$ then, for $0 \leq s \leq 1$, $\mathcal{H}^s(\Omega\!\to\!\Gamma) = H^s(\Gamma)$, and moreover, the identity  operator $\mathcal{H}^s(\Omega\!\to\!\Gamma) \to H^s(\Gamma)$ is a bounded isomrophism. 
The conclusions hold in particular for Lipschitz domains.
\end{theo}

Theorem \ref{interpol} does immediately imply the following.
\begin{cor} Under the assumption of  Theorem \ref{interpol} the following assertions hold:
\begin{enumerate}
    \item[\rm(i)] The identity operator among any two of $H^s(\Gamma)$, $\mathcal H^s(\Omega_i\!\to\!\Gamma)$ and $\mathcal H^s(\Omega_e\!\to\!\Gamma)$ is a bounded isomorphism with respect to $\|\cdot\|_{H^s(\Gamma)}$, $\|\cdot\|_{\mathcal H^s(\Omega_i)}$ and $\|\cdot\|_{\mathcal H^s(\Omega_e)}$. 
    \item[\rm(ii)] If $s \in (1/2, 1)$ then each $f \in H^s(\Gamma)$ is bounded and continuous. 
\end{enumerate}
    \end{cor}
\noindent 
We remark that assertion $\rm (ii)$ is consistent with the classical Sobolev embedding theorem saying that if $s > n/2$ then each $f \in H^s(\mathbb R^n)$ is bounded and continuous (see \cite{Tay1}).

\begin{proof}[Proof of Theorem \ref{interpol}]
Recall that 
Bergman spaces with standard weights $A^p_\alpha$ are defined as the sets of analytic functions $g$ on $\mathbb D$ such that $g \in L^p(\mathbb D, (1-|z|^2)^{\alpha}dxdy)$, i.e., 
$$
\|g\|_{A^p_{\alpha}}^p = \iint_{\mathbb D} |g(z)|^p (1-|z|^2)^{\alpha}dxdy < \infty
$$
where $p > 0$ and $\alpha > -1$. Here,  the assumption that $\alpha > -1$ is essential because the space $L^p(\mathbb D, (1-|z|^2)^{\alpha}dxdy)$ does not contain any holomorphic function other than $0$ when $\alpha \leq -1$. Zhao-Zhu  extended the definition of $A^p_{\alpha}$ to the case where $\alpha$ is any real number; that is consistent with the traditional definition when $\alpha > -1$ (see Theorem 13 in \cite{ZZ}). From the definition it can be seen that if $\alpha < \beta$ then the strict inclusion $A^p_{\alpha} \subset A^p_{\beta}$ holds. 

For our purpose we in particular mention that the space $A^2_{-1}$ is the classical Hardy space $E^2(\mathbb D)$, i.e., that is the set of analytic functions $g$ in $\mathbb D$ such that
$$
\int_{\mathbb D}|g'(z)|^2(1-|z|^2)dxdy < \infty.
$$
We also need to use the case of $p = 2$ and $\alpha = 1-2s$ where $0 \leq s < 1$, namely, 
$$
A^2_{1-2s} = \{g \;\text{analytic in}\; \mathbb D: \iint_{\mathbb D}|g(z)|^2 (1-|z|^2)^{1-2s} dxdy < \infty\}.
$$
This is a closed linear subspace of the Hilbert space $L^2(\mathbb D, (1-|z|^2)^{1-2s}dxdy)$. For these spaces, the following interpolation theorem (see Theorem 36 in \cite{ZZ}) holds true: 
\begin{equation}\label{Binterpol}
    [A^2_1, A^2_{-1}]_{s} = A^2_{1-2s},\qquad 0 < s < 1
\end{equation}
with equivalent norms.

Suppose $\varphi$ is a conformal map from $\mathbb D$ onto $\Omega$ fixing $0$. For any $f \in \mathcal A^s(\Omega)$, $0 \leq s < 1$, we have seen from \eqref{Vs} that
$$
\iint_\mathbb D |V_s(f)'(z)|^2 (1-|z|)^{1-2s} dxdy \approx \iint_{\Omega} |f'(\zeta)|^2 d(\zeta,\Gamma)^{1-2s} d\xi d\eta
$$
where the implicit constants depend only on $s$, and 
$$V_s(f)'(z)=(f\circ\varphi)'(z)\varphi'(z)^{1/2-s}.$$
From it we see that  
\begin{equation}\label{neq1}
    f \in \mathcal A^s(\Omega) \;\Leftrightarrow \; V_s(f)' \in A^2_{1-2s} 
\end{equation}
with comparable norms. 
Note that any $g \in A^2_{1-2s}$ can be written in the form $V_s(f)'$ by choosing $f(z) = \int_0^zg\circ\varphi^{-1}(u)(\varphi^{-1})'(u)^{3/2-s} du$. 

Motivated by  equivalence \eqref{neq1}, we define $\mathcal A^1(\Omega)$ as the set of analytic functions $f$ on $\Omega$ satisfying $V_1(f)' \in A^2_{-1}$. It is easy to see that $f$ is an anti-derivative of a function in $E^2(\Omega)$. To be more precise,  recall that $A^2_{-1} = E^2(\mathbb D)$ we have 
$$
\int_{\Gamma} |f'|^2 d\sigma = \int_{\mathbb T} |f'\circ\varphi(z) \varphi'(z)^{1/2}|^2 |dz| = \int_{\mathbb T} |V_1(f)'(z)|^2 |dz| < \infty.
$$
By this definition,  the space of the boundary trace of $\mathcal H^1(\Omega)$, the harmonic counterpart of $\mathcal A^1(\Omega)$, is just $H^1(\Gamma)$, i.e., $\mathcal H^1(\Omega\!\to\!\Gamma) = H^1(\Gamma)$. 

For any $u \in \mathcal H^s(\Omega)$, $0 \leq s \leq 1$,  set $f = u + i\tilde u$ so that $f\in\mathcal A^s(\Omega)$, and 
$V_s(f)'\in A^2_{1-2s}$. We define the linear operator $\Pi$ in $A^2_{1-2s}$ as 
$$
\Pi\left((f\circ\varphi)'(\varphi')^{1/2-s}\right) = u. 
$$
Suppose $s = 0$.  Since the identity operator from $\mathcal A^0(\Omega)$ onto $E^2(\Omega)$ is a bounded isomorphism by Theorem \ref{E2A0}, it holds that 
\begin{equation}\label{case0-2}
    \|(V_0(f))'\|_{A^2_1}\approx \|f\|_{\mathcal H^0(\Omega)}\approx \|f\|_{E^2(\Omega)}.
\end{equation} 
It follows from Theorem \ref{conj} that 
$$
\int_{\Gamma} |u|^2 d\sigma  \approx \int_{\Gamma} |f|^2 d\sigma.
$$
Then, we conclude that the linear operator $\Pi$ is  bounded from the Bergman space $A^2_1$ to $L^2(\Gamma, d\sigma)$, and by Corollary \ref{0trace} that it is surjective.  Suppose $s = 1$. Since
$$
\|u\|_{H^{1}(\Gamma)} \approx \|f\|_{H^{1}(\Gamma)} \approx \|f\|_{\mathcal H^{1}(\Omega)} \approx \|V_{1}(f)'\|_{A^2_{-1}},
$$
where the first ``$\approx$'' is still due to Theorem \ref{conj}, 
the linear operator $\Pi$ is surjective and bounded from the Bergman space $A^2_{-1}$ to $H^1(\Gamma)$.

Notice that for any $u \in \mathcal H^s(\Omega)$, since $|\nabla u| = |\nabla \tilde u|$ we have
$$
\|u\|_{\mathcal H^s(\Omega)} \approx \|f\|_{\mathcal H^s(\Omega)} \approx \|V_s(f)'\|_{A^2_{1-2s}}.
$$
Now we can invoke the interpolation theorem of Bergman spaces \eqref{Binterpol} and Calder\'on's interpolation theorem (see Theorem \ref{Cald}). It follows from the  functorial property of complex interpolation that the linear operator 
$\Pi$ maps $A^2_{1-2s}$ bounded to $H^s(\Gamma)$ for $0 < s < 1$, and moreover, for any $u \in \mathcal H^s(\Omega)$, $\|u\|_{H^s(\Gamma)} \lesssim \|V_s(f)'\|_{A^2_{1-2s}}$, and thus,  $\|u\|_{H^s(\Gamma)} \lesssim \|u\|_{\mathcal H^s(\Omega)}$. 
If we apply the  functorial property of complex interpolation to the linear operator $\Pi^{-1}$ mapping from 
$L^2(\Gamma, d\sigma)$ to $A^2_1$, then we get that $\Pi^{-1}$ maps $H^s(\Gamma)$ bounded to $A^2_{1-2s}$, and for any $u \in H^s(\Gamma)$, 
$$
\|u\|_{\mathcal H^s(\Omega)}  \approx \|V_s(f)'\|_{A^2_{1-2s}} \lesssim \|u\|_{H^s(\Gamma)}. 
$$ 
This completes the proof of Theorem \ref{interpol}. 
\end{proof}

\begin{rem}
    The above proof of Theorem \ref{interpol} makes full use of the interpolation theorem of Bergman spaces, which was proved using Bergman type projections $P$ from $L^2(\mathbb D, (1-|z|^2)^\alpha dxdy)$ to $A^2_{\alpha}$, and the Stein-Weiss interpolation theorem: for any $\theta \in (0, 1)$, 
    $$
    [L^2(\mathbb D, \omega_0(z)dxdy), L^2(\mathbb D, \omega_1(z)dxdy)]_\theta = L^2(\mathbb D, \omega(z)dxdy)
    $$
    with equal norms, provided that the weight functions have the relation:  $\omega = \omega_0^{1-\theta}\omega_1^{\theta}$. A theorem of Bekolle-Bonami \cite{BB} says that $P$ extends to $L^2(\mathbb D,\omega(z)dxdy)$ if and only if the weight function $\omega\in B_2(\mathbb D)$, where $B_2(\mathbb D)$ is like the Muckenhoupt class $A_2$ on $\mathbb D$, but only for the Carleson squares. 
    
    It is easy to check that if $|\varphi'|\in A_2$ on $\mathbb T$ then $ \omega(z) = \left((1-|z|)|\varphi'(z)|\right)^{1-2s}\in B_2(\mathbb D)$ for all $s\in [0,1].$ On the other hand, by simple computation we have $f \in \mathcal A^s(\Omega)$ if and only if $F = f\circ\varphi$ satisfies that 
    $$
    \int_{\mathbb D}|F'(z)|^2 \omega(z) dxdy < \infty.
    $$
    In other words, $F'$ belongs to the weighted Bergman space $A^2_\omega$. Given these observations, it is possible to give a different proof of Theorem \ref{interpol} using  generalized Bergman projections and the Stein-Weiss interpolation theorem. We will not go into further details here. 
\end{rem}

\bigskip

\textbf{Acknowledgments. }
This work is supported by the National Natural Science Foundation of China (Grant No. 12271218).

\bibliographystyle{abbrv}
\bibliography{ref}

\begin{thebibliography}{10}

\bibitem{Ahl}
Ahlfors, L.V.: 
\newblock {\em Lectures on Quasiconformal Mappings}. 
\newblock Mathematical Studies, vol. 10. Van Nostrand, Princeton (1966)

\bibitem{AS}
Aronszajn, N., Smith, K.T.: 
\newblock Theory of Bessel potentials. Part I. 
\newblock {\em Ann. Inst. Fourier, Grenoble}  11, 385-475 (1961)

\bibitem{Ast}
Astala, K.: 
\newblock Calder\'on's problem for Lipschitz classes and the dimension of quasicircles. 
\newblock {\em Revista Mat. Iberoamericana}  4, 469-486 (1988)

 \bibitem{BB}
B\'ekoll\'e, D., Bonami, A.: 
\newblock In\'egalit\'es \`a poids pour le noyau de Bergman. 
\newblock {\em C. R. Acad. Sci. Paris, S\'er.  A}  286,   775-778 (1978)

\bibitem{BL}
Bergh, J., L\"ofstr\"om, J.: 
\newblock {\em Interpolation Spaces. An Introduction}. 
\newblock Grundlehren der Mathematischen Wissenschaften, No. 223. Springer-Verlag, Berlin-New York (1976)

\bibitem{BMMM} Brewster, K., Mitrea, D., Mitrea, I., Mitrea, M.: 
\newblock Extending Sobolev functions with partially vanishing traces from locally $(\epsilon, \delta)$-domains and applications to mixed boundary problems.
\newblock {\em J.  Funct.  Anal.}  266(7),  (2014)

\bibitem{Cal}
Calder\'on, A.P.: 
\newblock Cauchy integrals on Lipschitz curves and related operators. 
\newblock {\em Proc. Nat. Acad. Sci. U.S.A.}  74, 1324-1327 (1977)

\bibitem{Cal1}
Calder\'on, A.P.: 
\newblock Intermediate spaces and interpolation. 
\newblock {\em Studia Math., Special Series 1}, 31-34 (1963)

\bibitem{CoF}
Coifman, R., Fefferman, C.: 
\newblock Weighted norm inequalities for maximal functions and singular integrals. 
\newblock {\em Studia Math.} 51, 241-250 (1974) 

\bibitem{CMM}
Coifman, R.,  McIntosh, A., Meyer, Y.:  
\newblock L'int\'egrale de Cauchy d\'efinit un op\'erateur born\'e sur $L^2$ pour les courbes lipschitziennes. 
\newblock {\em Ann. of Math.} 116(2), 361-387 (1982)

\bibitem{Dav}
David, G.: 
\newblock Op\'erateurs int\'egraux  singuliers sur certaines courbes du plan complexe. 
\newblock {\em Ann. Sci. \'Ecole Norm. Sup.} 17, 157-189 (1984)

\bibitem{Dav87}
David, G.: 
\newblock A lower bound for the norm of the Cauchy operator on Lipschitz graphs. 
\newblock {\em Trans. Amer. Math. Soc.} 302(2), 741–750 (1987)


\bibitem{Din}
Ding, Z.: 
\newblock A proof of the trace theorem of Sobolev spaces on Lipschitz domains. 
\newblock {\em Proc. Amer. Math. Soc.} 124(2), 591-600 (1996)

\bibitem{Gar}
Garnett, J.: 
\newblock {\em Bounded Analytic Functions}. 
\newblock Academic Press, New York (1980)


\bibitem{GM}
Gehring, F.W.,  Martio, O.:  
\newblock Quasidisks and the Hardy-Littlewood property. 
\newblock {\em Complex Variables Theory Appl.}  2, 67-78 (1983)

\bibitem{GW}
Gundy, R.F., Wheeden, R.L.: 
\newblock Weighted integral inequalities for the nontangential maximal function, Lusin area integral, and Walsh-Paley series. 
\newblock {\em Studia Math.} 49, 107–124 (1973/74)


\bibitem{JK}
Jerison, D.S.,  Kenig, C.E.: 
\newblock Hardy spaces, $A_{\infty}$, and singular integrals on chord-arc domains. 
\newblock {\em Math. Scand.} 50, 221-247 (1982)

\bibitem{Jon}
Jones, P.W.: 
\newblock Quasiconformal mappings and extendability of functions in Sobolev spaces. 
\newblock {\em Acta. Math.} 147,  71-88 (1981)


\bibitem{JW1} 
Jonsson, A., Wallin, H.: 
\newblock {\em Function Spaces on Subsets of $\mathbb R^n$.}
\newblock {\em Math. Rep.}  2(1)  (1984)


\bibitem{JoZi} 
Jones, P.,  Zinsmeister, M.:  
\newblock Sur la transformation conforme des domaines de Lavrentiev. 
\newblock {\em C. R. Acad. Sci. Paris S\'er. I Math.}  295(10), 563-566 (1982)


\bibitem{LS} 
Liu, T.,  Shen, Y.: 
\newblock The jump problem for the critical Besov space. 
\newblock {\em Math. Z.}  306(4), 59 (2024)


\bibitem{Mur} 
Murai, T.:  
\newblock Boundedness of singular integral operators of Calder\'on type. VI, 
\newblock {\em Nagoya Math. J.} 102, 127-133 (1986)

\bibitem{NaS} 
Nag, S., Sullivan, D.: 
\newblock Teichm\"uller theory and the universal period mapping via quantum calculus and the  $H^{1/2}$  space on the circle. 
\newblock {\em Osaka J. Math.} 32(1), 1-34 (1995)

\bibitem{Ple} 
Plemelj, J.:  
\newblock Riemannsche Funktionenscharen mit gegebener Monodromiegruppe. 
\newblock {\em Monatsh.  Math. Phys.} 19(1), 211-245 (1908)

\bibitem{Sc1} 
Schippers,  E., Staubach, W.:  
\newblock Harmonic reflection in quasicircles and well-posedness of a Riemann problem on quasidisks. 
\newblock {\em J. Math. Anal. Appl.} 448(2), 864–884 (2017)


\bibitem{Ste} 
Stegenga, D.A.:  
\newblock Multipliers of the Dirichlet space. 
\newblock {\em Illinois J. Math.}  24(1), 113-139 (1980)  

\bibitem{Tay} 
Taylor, G.D.:  
\newblock Multiplies on $D_{\alpha}$. 
\newblock {\em Trans. Amer. Math. Soc.} 123, 229–240 (1966)

\bibitem{Tay1} 
Taylor, M.E.: 
\newblock {\em Partial Differential Equations I.  Basic theory}.  
\newblock  Appl. Math. Sci., 115. 
Springer, Cham (2023)


\bibitem{Tri} 
Triebel, H.: 
\newblock {\em Theory of Function Spaces. II.} 
\newblock Monogr. Math., 84, Birkh\"auser Verlag, Basel (1992)


\bibitem{Va} 
V\"ais\"al\"a, J.:  
\newblock Porous sets and quasisymmetric maps. 
\newblock {\em Trans. Amer. Math. Soc.}, 299(2), 525-533 (1987)

\bibitem{VW} 
Viklund, F., Wang, Y.: 
\newblock Interplay between Loewner and Dirichlet energies via conformal welding and flow-lines. 
\newblock {\em Geom. Funct. Anal.}  30, 289-321 (2020)

\bibitem{WZ} 
Wei, H., Zinsmeister, M.: 
\newblock Dirichlet spaces over chord-arc domains. 
\newblock {\em Math. Ann.}  391(1), 
 1045–1064 (2025)


\bibitem{ZZ} 
Zhao, R., Zhu, K.: 
\newblock Theory of Bergman Spaces in the Unit Ball of $\mathbb C^n$. 
\newblock {\em M\'emoires SMF}, 2006


\bibitem{Zi1} 
Zinsmeister, M.:  
\newblock Problèmes de Dirichlet, Neumann, Calderón dans les quasidisques pour les classes höldériennes 
\newblock (Dirichlet, Neumann and Calder\'on problems in quasidisks for H\"older classes).  
\newblock {\em Rev. Mat. Iberoamericana} 2(3), 319–332 (1986)



\end{thebibliography}

\end{document}